\documentclass[12pt]{amsart}
\usepackage{amsfonts}
\usepackage{enumerate}
\usepackage{helvet,courier}
\usepackage{mathptmx,amsmath}
\usepackage{type1cm}
\DeclareMathAlphabet{\mathcal}{OMS}{cmsy}{m}{n}
\DeclareMathAlphabet{\mathbbold}{U}{bbold}{m}{n}  
\usepackage{amsthm}
\usepackage{graphicx}       
\usepackage{multicol}
\usepackage{mathrsfs}
\usepackage{amssymb}
\usepackage{verbatim}
\usepackage[usenames,dvipsnames]{xcolor}
 
\theoremstyle{plain}
\newtheorem{thm}{Theorem}

\newtheorem{lm}[thm]{Lemma}
\newtheorem{lemma}[thm]{Lemma}
\newtheorem{cor}[thm]{Corollary}
\newtheorem{prop}[thm]{Proposition}

\newtheorem{conj}{Conjecture}
\newcounter{question}
\newcommand{\qt}{%
        \stepcounter{question}%
        \thequestion}
\newcommand{\bq}{\fbox{Q\qt}\ }

\newcounter{typo}

\theoremstyle{remark}
\newtheorem{rmk}{Remark}

\theoremstyle{definition}

\newcommand{\bnu}{\begin{enumerate}}
\newcommand{\enu}{\end{enumerate}}
\newcommand{\bpf}{\begin{proof}}
\newcommand{\epf}{\end{proof}}

\newcommand{\qq}{\qquad}

\newcommand{\al}{\alpha}

\newcommand{\ga}{\gamma}

\newcommand{\om}{\omega}

\newcommand{\la}{\lambda}

\newcommand{\ep}{\epsilon}

\newcommand{\si}{\sigma}

\newcommand{\vp}{\varphi}

\newcommand{\bbr}{\mathbb{R}}

\newcommand{\rn}{\mathbb{R}^n}
\newcommand{\rrn}{\mathbb{R}^n}

\newcommand{\f}{\frac}

\newcommand{\nf}{\infty}

\newcommand{\tf}{\tfrac}
\newcommand{\wh}{\widehat}

\newcommand{\supp}{\mathrm{supp}}
\newcommand{\norm}[1]{\left\Vert{#1}\right\Vert}
\usepackage{amsxtra}    

\newcommand{\vu}{{\vec\mu}}

\allowdisplaybreaks

\makeindex         

\begin{document}

\title[The H\"ormander multiplier theorem]{The H\"ormander multiplier theorem, II: The bilinear local $L^2$ case}
\author{Loukas Grafakos}

\address{Department of Mathematics, University of Missouri, Columbia MO 65211, USA}
\email{grafakosl@missouri.edu}

\author{Danqing He}

\address{Department of Mathematics, 
Sun Yat-sen (Zhongshan) University, 
Guangzhou, 510275, 
P.R. China}
\email{hedanqing35@gmail.com}

\author[Honzik]{Petr Honzik}
\address{Faculty of Mathematics and Physics, Charles University in Prague,
Ke Karlovu 3,
121 16 Praha 2, Czech Republic}
\email{honzik@gmail.com}

\thanks{The first author was supported by the Simons Foundation. The third author  
 was supported by the ERC CZ grant LL1203 of the Czech Ministry of Education}

\thanks{{\it Mathematics Subject Classification:} Primary   42B15. Secondary 42B25}
\thanks{{\it Keywords and phases:} Bilinear multipliers, H\"ormander multipliers,
wavelets, multilinear operators}
\date{}
\begin{abstract}
We use wavelets of   tensor product type  to obtain the boundedness of bilinear multiplier operators
on $\mathbb R^n\times \mathbb R^n$ associated with  
H\"ormander multipliers  on $\mathbb R^{2n}$ with 
minimal smoothness. We focus on the local $L^2$ case and we obtain boundedness under the minimal smoothness assumption of $n/2$ derivatives. We also provide counterexamples to obtain necessary
  conditions   for all sets of indices.
\end{abstract}

\maketitle

\section{Introduction}

An $m$-linear $(p_1,\dots, p_m,p)$ multiplier $\si(\xi_1,\dots,\xi_m)$ is a   function on $\mathbb R^{ n}\times \cdots \times \mathbb R^{ n}$
such that the corresponding $m$-linear operator  
$$
T_\si(f_1,\dots , f_m)(x)= \int_{\mathbb R^{mn}}\si(\xi_1,\dots,\xi_m)\wh f_1(\xi_1)\cdots 
\wh f_m(\xi_m) e^{2\pi i x\cdot(\xi_1+\cdots+\xi_m)}d\xi_1\cdots d\xi_m, 
$$
 initially defined on $m$-tuples of Schwartz functions, 
has a  bounded extension 
from $L^{p_1}(\mathbb R^n)\times\cdots\times L^{p_m}(\mathbb R^n)$
to $L^{p}(\mathbb R^n)$ for appropriate $p_1,\dots, p_m,p$. 

It is known from the work in \cite{CM} for $p>1$ and  ~\cite{KS},~\cite{GT} for $p\le 1$, that the classical Mihlin
condition on $\si$ in $\mathbb R^{mn}$ yields boundedness for $T_\si$ from $L^{p_1}(\mathbb R^n)\times\cdots\times L^{p_m}(\mathbb R^n)$
to $L^{p}(\mathbb R^n)$ for all $1<p_1,\dots p_m\le \infty$, $1/m<p =  (1/p_1+\cdots +1/p_m)^{-1}<\infty$. 
The Mihlin condition in this setting is usually referred to as the Coifman-Meyer 
condition and  the associated multipliers bear the same names as well. The Coifman-Meyer condition cannot be 
weakened to the  
Marcinkiewicz condition, as the latter fails in the multilinear setting; see \cite{GK}. 
Related multilinear multiplier theorems  
with mixed smoothness (but not necessarily minimal) can  be found in \cite{MT}, \cite{MPTT1},
 \cite{GHNY}. 

A natural question on H\"ormander type multipliers is how the minimal smoothness
 $s $ interplays with the  
 range of $p$'s on which boundedness is expected.   In the linear case, this question
was studied in \cite{CT},   \cite{Seeger}, and \cite{P1}. 
Let $L^r_s(\rn)$ be the Sobolev space consisting of all functions $h$
such that
$(I-\Delta)^{s/2}(h)\in L^r(\rn)$, where $\Delta$ is the Laplacian.
In the first paper of this series 
\cite{P1}, we showed that the conditions  $|1/2-1/p|< s/n$  and $rs>n$ 
imply $L^p(\mathbb R^n)$ boundedness for $1<p<\infty$ for $T_\si$ in the linear case $m=1$, when the 
 multiplier $\si$ lies   in the Sobolev space $L^r_s(\rn)$ 
uniformly over all annuli. 
This minimal smoothness problem in the bilinear setting was first studied in \cite{T} and later in \cite{MT} and \cite{GMT}. These references  contain necessary conditions on $s$ when the multiplier   in the Sobolev space $L^r_s$ with $r=2$; 
other values of $r$ were considered in \cite{GS}. 

Our goal here is to pursue the analogous bilinear question. In this   paper we 
focus on the  boundedness of $T_\si$ in the local $L^2$ case, i.e., the situation where $1\le p_1,p_2\le 2$ and $1\le p= 1/(1/p_1+1/p_2)\le 2$ under minimal smoothness conditions on $s$. It turns out that 
to express our result in an optimal fashion, we need to work with 
 $r>2$.  We also work with the   case $L^2\times L^2\to L^1$ as  boundedness in the remaining local $L^2$ 
 indices follows by duality and interpolation. 
 We achieve our goal via new technique to study 
 boundedness for bilinear operators based on
 tensor product wavelet decomposition     developed in \cite{GHH}; this technique was recently used to
 solve other problems; see \cite{He}.

The main   result  of this paper is the following theorem.

\begin{thm}\label{MR2} Suppose $\wh\psi\in\mathcal C_0^{\nf}(\bbr^{2n})$ is positive and  supported in the annulus 
$\{(\xi,\eta):1/2\le|(\xi,\eta)|\le 2\}$ such that $\sum_{j\in\mathbb Z}\wh\psi_j(\xi,\eta)=
\sum_{j}\wh\psi(2^{-j}(\xi,\eta))=1$ for $(\xi,\eta)\neq0$.
Let $ 1<r<\infty$, $s>\max\{n/2, 2n/r\}$, and suppose  there is a constant $A$ such that
\begin{equation}\label{usb}
\sup_{j}\|\si(2^j\cdot)\wh\psi\|_{L^r_s(\bbr^{2n})}\le A<\nf.
\end{equation}
Then there is a constant $C=C(n,\Psi)$ such that  the bilinear operator 
$$
T_{\si}(f,g)(x)=
\int_{\bbr^{2n}}\si(\xi,\eta)\wh f(\xi)\wh g(\eta)e^{2\pi i x\cdot (\xi+\eta)}d\xi d\eta,
$$
initially defined on   Schwartz functions $f$ and $g$,
satisfies   
\begin{equation}
\|T_{\si}(f,g)\|_{L^1(\rn)}\le CA\|f\|_{L^2(\rn)}\|g\|_{L^2(\rn)}.
\end{equation}
\end{thm}

The optimality of \eqref{usb} in the preceding theorem is contained in the following result. 
\begin{thm}\label{MR3}
Suppose that for  $0<p_1,\dots , p_m<\infty$,  $p=(1/p_1+\cdots +1/p_m)^{-1}$,  we have
\begin{equation} \label{B1}  
\|T_\si\|_{L^{p_1}(\bbr^n)\times\cdots\times L^{p_m}(\bbr^n)\to L^p(\bbr^n)}\le C\sup_{j\in\mathbb Z}\|\si(2^j\cdot)\wh\Psi\|_{L^r_s(\bbr^{mn})}  
\end{equation}
for all bounded  functions $\si$ for which   $\sup_{j\in\mathbb Z}\|\si(2^j\cdot)\wh\Psi\|_{L^r_s(\bbr^{mn})}<\nf$ (for some fixed $r,s>0$).  
Then we must necessarily have $s\ge \max\{(m-1)n/2,mn/r\}$.
\end{thm}

Finally, we have another set of necessary conditions for the boundedness of $m$-linear multipliers.  
The sufficiency of these conditions is shown in the third paper of this series. 

\begin{thm}\label{PL} Suppose 
there exists a constant $C$ such that \eqref{B1} holds for   all $\si$ such that the right hand side is finite.
Then we must necessarily have 
$$
\f 1 p-\f 1 2\le \f s n+\sum_{i\in I}\Big(\f 1{p_i}-\f 1 2\Big),
$$
where $I$ is an arbitrary subset of $\{1,2,\dots,\ m\}$ which may also  be empty (in which case the sum is supposed to be zero).
\end{thm}

\section{Preliminaries}
We utilize wavelets with compact supports. 
Their existence is due to Daubechies \cite{Dau} and their construction is contained  in Meyer's book \cite{Meyer1} and Daubechies' book \cite{DauB}. 
For our purposes we need
product type smooth wavelets with compact supports; the construction  of  
such objects   we use here can be found in Triebel~\cite[Proposition 1.53]{Tr1}.   

\begin{lm}\label{TrDau}
For any fixed $k\in \mathbb N$ there exist real compactly supported functions $\psi_F,\psi_M\in \mathcal C^k(\mathbb R)$, the class of functions with continuous derivatives of
order up to $k$,
which
satisfy that $\|\psi_F\|_{L^2(\mathbb R)}=\|\psi_M\|_{L^2(\mathbb R)}=1$
and $\int_{\mathbb R}x^{\al}\psi_M(x)dx=0$ for $0\le\al\le k$,
such that,  if $\Psi^G$ is defined by 
$$
\Psi^{G}(\vec x\,)=\psi_{G_1}(x_1)\cdots \psi_{G_{2n}}(x_{2n}) 
$$
for   $G=(G_1,\dots, G_{2n})$ in the set     
$$
 \mathcal I :=\Big\{ (G_1,\dots, G_{2n}):\,\, G_i \in \{F,M\}\Big\}  \, , 
$$
then the  family of 
functions 
$$
\bigcup_{\vec \mu \in  \mathbb Z^{2n}}\bigg[  \Big\{   \Psi^{(F,\dots, F)} (\vec  x-\vec \mu  )  \Big\} \cup \bigcup_{\la=0}^\nf
\Big\{  2^{\la n}\Psi^{G} (2^{\la}\vec x-\vec \mu):\,\, G\in \mathcal I\setminus \{(F,\dots , F)\}  \Big\}  
  \bigg]
$$
forms an orthonormal basis of $L^2(\mathbb R^{2n})$, where $\vec x= (x_1, \dots , x_{2n})$.  
\end{lm}

In order to prove  our results, we use the wavelet characterization of Sobolev spaces, following Triebel's book \cite{Tr1}. Let us fix the smoothness $s,$ for our purposes we always have $s\leq n+1.$ Also, we only work with spaces with the integrability index $r>1.$ 
Take $\vp$ as a smooth function defined on $\mathbb R^{2n}$ such that $\wh\vp$ is supported in
the unit annulus such that $\sum_{j=0}^\nf\wh \vp_j=1$, where
$\wh\vp_j=\wh\vp(2^{-j}\cdot)$ for $j\ge1$ and $\wh\vp_0=\sum_{k\le 0}\wh\vp(2^{-k}\cdot)$.
Then  for a distribution $f\in\mathcal S'(\mathbb R^{2n})$  we define the $F^s_{r,q}$ norm as follows: 
$$
\|f|F^s_{r,q}(\mathbb R^{2n})\|=
\Big\|\big(\sum_{j=0}^\nf2^{jsq}|(\vp_j\wh f\,)^{\vee}(\cdot)|^q\big)^{1/q}\Big\|_{L^r(\mathbb R^{2n})}.
$$
We then pick   wavelets with smoothness and cancellation degrees $k=6n.$ This number suffices for the purposes 
of the following  lemma.

\begin{lm}[{\cite[Theorem 1.64]{Tr1}}]\label{TrSo}
Let $0<r<\nf,\ 0<q\le\nf, \ s\in\mathbb R$, and
for $\la\in \mathbb N$ and $\vec\mu\in\mathbb N^{2n}$ let $\chi_{\la\vec\mu}$
be the characteristic function of the cube $Q_{\la\vec\mu}$ centered at $2^{-\la}\vec\mu$ with length
$2^{1-\la}$.  For a sequence $\ga=\{\ga^{\la,G}_{\vec\mu}\}$ define the norm
$$
\|\ga|f^s_{r,q}\|=
\Big\|(\sum_{\la,G,\vec\mu}2^{\la sq}|\ga^{\la,G}_{\vec\mu}\chi_{\la\vec\mu}(\cdot)|)^{q/2}\Big\|_{L^r(\mathbb R^{2n})}.
$$

Let $\mathbb N\ni k>\max
\{s,\f{4n}{\min(r,q)}+n-s\}$.
Let $\Psi_{\vec\mu}^{\la,G}$ be the $2n$-dimensional Daubechies wavelets with smoothness $k$
according Lemma \ref{TrDau}.
Let $f\in \mathcal S'(\mathbb R^{2n})$. Then $f\in F^s_{r,q}(\mathbb R^{2n})$ if and only if
it can be represented as 
$$
f=\sum_{\la,G,\vec\mu}\ga^{\la,G}_{\vec\mu}2^{-\la n}\Psi^{\la,G}_{\vec\mu}
$$
with $\|\ga|f_{rq}^s\|<\nf$ with 
unconditional convergence   in $\mathcal S'(\rn)$. Furthermore this representation
is unique,
$$
\ga_{\vec\mu}^{\la,G}=2^{\la n}\langle f,\Psi^{\la,G}_{\vec\mu}\rangle,
$$
and
$$
I: f\to\big\{2^{\la n}\langle f,\Psi^{\la,G}_{\vec\mu}\rangle\big\}
$$
is an isomorphic map of $F^s_{r,q}(\mathbb R^{2n})$
onto $f^s_{r,q}.$

\end{lm}

In particular, the Sobolev space $L^r_s(\mathbb R^{2n})$ coincides with $F_{r,2}^s(\mathbb R^{2n})$. In the proof of our results, we use  for fixed $\la$ the following estimate:
\begin{equation}\label{1L}
\Big\|\Big(\sum_{G,\vec\mu }|\langle \si, \Psi^{\lambda,G}_{\vu} \rangle \Psi^{\lambda,G}_{\vu}|^2\Big)^{1/2}\Big\|_{L^r}\leq 
C\|\si\|_{L^r_s}2^{-s\lambda}.
\end{equation}
To verify this,  by Lemma \ref{TrSo}, 
we have 
$$
\bigg\|\sum_{G,\vu}2^{\lambda s}|\gamma^{\lambda,G}_\vu \chi_{Q_{\lambda,\vu}}|\bigg\|_{L^r} \leq C\|\si\|_{L^r_s},
$$
with  $\gamma^{\lambda,G}_\vu=2^{\lambda n}\langle \si, \Psi^{\lambda,G}_\vu\rangle.$ 
Notice that $2^{-\lambda n}\Psi^{\lambda,G}_{\vu}$ are $L^{\infty}$ normalized wavelets, and there
exists an absolute  constant $B$ such that the support of 
$\Psi^{\lambda,G}_{\vu}$ is always contained in $\cup_{|\vec\nu|\le B} Q_{\lambda,\vu+\vec\nu}$.
This then implies~\eqref{1L}.

\section{The main lemma}\label{MLS}
 Let $Q$ denote  the  cube $[-2,2]^{2n}$ in $\mathbb R^{2n}$, and consider a Sobolev space  $L^r_s(Q)$ as the Sobolev space 
 of distributions supported in  $Q $ which are in $L^r_s(\mathbb R^{2n})$.

\begin{lemma}\label{MR}
For $r\in (1,\infty)$ let $s>\max(n/2,2n/r)$ and suppose $\si \in L^r_s(Q).$ 
Then $\si$ is a bilinear multiplier  bounded   from $L^2(\rn)\times L^2(\rn)$ to $L^1(\rn)$.
\end{lemma}

\begin{proof}
The important inequality is the one for a single generation of   wavelets (with $\lambda$ fixed). 
For   a fixed $\la$, by the uniform compact supports of the elements in 
the basis, we can classify the wavelets into finitely many subclasses such that the supports
of the elements in each subclass are pairwise disjoint. We denote by
 $D_{\lambda,\kappa}$ such a subclass and   the related symbol 
$$
\si_{\lambda,\kappa}=\sum_{\omega\in D_{\lambda,\kappa}}a_\omega \omega,
$$
where $a_\omega=\langle \si, \omega \rangle .$ The $\om$'s are $L^2$ normalized, but we change the normalization to $L^r,$ i.e. we consider $\tilde \omega = \omega/\|\omega\|_{L^r}$ and $b_\omega=a_\omega\|\omega\|_{L^r}.$ We have      
$$
\si_{\lambda,\kappa}=\sum_{\omega\in D_{\lambda,\kappa}}b_\omega \tilde \omega
$$
and from the Sobolev smoothness and the fact that the supports of the wavelets do not overlap, with the aid of   \eqref{1L} we obtain  
\begin{align*}
B   =\bigg( \sum_{\omega\in D_{\lambda,\kappa}} |b_\omega|^r\bigg)^{1/r}  
& = \bigg(\sum_{\om}\int \Big(|a_\om\om|^2\Big)^{r/2}dx\bigg)^{1/r} \\
& \le \Big\|\Big(\sum_\om |a_\om\om|^2\Big)^{1/2}\Big\|_{L^r}\\
&  \leq C \|\si\|_{L^r_s} 2^{-s\lambda}.
\end{align*} 
Now, each  $\omega$ in $D_{\la,\kappa}$ is of the form $\omega=\omega_k\omega_l$ 
with $\vu=(k,l)$, where $k$ and $l$ both range over  index sets $U_1$ and $U_2$ of cardinality at most $C2^{\lambda n}.$ Moreover we denote by $b_{kl}$ the coefficient $b_\om$, and
we have
$$ \si_{\lambda,\kappa}=\sum_{k\in U_1}\tilde \omega_k\sum_{l\in U_2} b_{k l} \tilde \omega_l.$$

Set $\tau_{\rm max}$ to be  the positive number such that $2n\lambda/r\le 
\tau_{\max}< 1+2n\lambda/r$.
For a nonnegative number $  \tau<2n\lambda/r=\tau_{\rm max}$  and a positive constant (depending on $\tau$) $K=2^{\tau r/2}$ we introduce the following decomposition: We define
the level set according to $b$ as
$$
D_{\lambda,\kappa}^{\tau}=\{\omega\in D_{\lambda,\kappa}: B2^{-\tau}<|b_\omega|\leq  B2^{-\tau+1}\},
$$
when $\tau<\tau_{\max}$.
We  also define the  set 
$$
D_{\lambda,\kappa}^{\tau_{\rm max}}= \{\omega\in D_{\lambda,\kappa}: |b_\omega|\leq  B2^{-\tau_{\rm max}+1}\}.  
$$
We now take the part with heavy columns
$$
D_{\lambda,\kappa}^{\tau,1}=\{\omega_k\omega_l\in 
D_{\lambda,\kappa}^{\tau}: 
{\rm card}\{s:\omega_k\omega_s\in D_{\lambda,\kappa}^{\tau}\}\geq K\},
$$
and the remainder  
$$
D_{\lambda,\kappa}^{\tau,2} = D_{\lambda,\kappa}^{\tau}\setminus D_{\lambda,\kappa}^{\tau,1}.
$$
We also use the following notations for the index sets: $U_1^{\tau,1}$  is the set of $k$'s
such that $\om_k\om_l$ in $D_{\lambda,\kappa}^{\tau,1}$, and for each $k\in U_1^{\tau,1}$ 
we denote $U_{2,k}^{\tau,1}$ the set of corresponding second indices
$l$'s such that $\om_k\om_l \in D_{\lambda,\kappa}^{\tau,1}$, whose cardinality is 
at least $K$. 
We also denote $$ \si_{\lambda,\kappa}^{\tau,1}=\sum_{k\in U_1^{\tau,1}}\tilde \omega_k\sum_{l\in U_{2,k}^{\tau,1}} b_{kl} \tilde \omega_l,$$ thus summing over the wavelets in the set $D_{\lambda,\kappa}^{\tau,1}.$
The symbol  $\si_{\lambda,\kappa}^{\tau,2}$ is then defined by summation over $D_{\lambda,\kappa}^{\tau,2}.$

We first treat the part $\si_{\lambda,\kappa}^{\tau,1}.$ 
Denote $\gamma={\rm card  }\  U_1^{\tau,1}$. For $\tau<\tau_{\rm max}$ the 
$\ell^r$-norm of the part of the sequence $\{b_{kl}\}$ indexed by the set $D_{\lambda,\kappa}^{\tau,1}$ 
is comparable to 
$$
C \bigg(\sum_{k\in U_1^{\tau,1}}\sum_{l\in U_{2,k}^{\tau,1}} (B2^{-\tau})^r\bigg)^{1/r}
$$
which is at least as big as 
$C(\gamma K (B 2^{-\tau})^r)^{1/r}$. However this $\ell^r$-norm is smaller than $B$, therefore we get $\gamma\leq C2^{\tau r}/ K= C2^{\tau r/2}.$ For  $\tau=\tau_{\rm max}$ we  trivially have that
 $\gamma\le C2^{n\lambda}= C2^{\tau_{\max} r/2}.$

For  $f,g \in \mathcal S$ we estimate the multiplier norm of $\si_{\lambda,\kappa}^{\tau,1}$ as follows:
$$ \begin{aligned}
\|{\mathcal F}^{-1}(\si_{\lambda,\kappa}^{\tau,1}\widehat f \widehat g\,)\|_{L^1}&\leq \sum_{k\in U_1^{\tau,1}}\| \widehat f \tilde \omega_k\|_{L^2}\| \sum_{l\in U_{2,k}^{\tau,1}}  b_{kl} \tilde \omega_l \widehat g\,\|_{L^2} \\&\leq  C \sum_{k\in U_1^{\tau,1}}\| \widehat f \tilde \omega_k\|_{L^2} \sup_{l}|b_{kl}|2^{\lambda n/r}\|g\|_{L^2} \\
&\leq C2^{\lambda n/r}\|g\|_{L^2}
\bigg(\sum_k\sup_{l}|b_{kl}|^2\bigg)^{1/2} \bigg(\sum_{k}\| \widehat f \tilde\omega_k\|_{L^2}^2 \bigg)^{1/2}.
\end{aligned}
$$
In view of orthogonality and of the fact that $\|\tilde \omega_k\|_{L^{\nf}}\approx 2^{\lambda n/r}$ we obtain 
the inequality
$$ 
\bigg(\sum_{k}\| \widehat f \tilde\omega_k\|_{L^2}^2 \bigg)^{1/2} \leq C 2^{\f {\lambda n}r}\|f\|_{L^2}.
$$
By the definition of $U_1^{\tau,1}$ we have also that
$$
\bigg(\sum_k\sup_{l}|b_{kl}|^2\bigg)^{1/2}  \leq   B 2^{-\tau}\gamma^{\f 12} .
$$
Collecting these estimates, we deduce 
\begin{equation}\label{rows}
\|{\mathcal F}^{-1}(\si_{\lambda,\kappa}^{\tau,1}\widehat f \,\,\widehat g\,)\|_{L^1}\leq 
C\|\si\|_{L^r_s}\gamma^{\f 1 2}2^{\la( \f{2n}r-s)} 2^{-\tau}\|f\|_{L^2}\|g\|_{L^2} .
\end{equation}

The  set $D_{\lambda,\kappa}^{\tau,2}$ has   the property that in each column there are at most $K$  elements. Let us denote by   $V^2$ the index set of all second indices 
such that $\tilde\om_k\tilde\om_l\in
D_{\lambda,\kappa}^{\tau,2}$, and for each $l\in V^2$ set 
$V^{1,l}$ the corresponding sets of first indices.  Thus $$D_{\lambda,\kappa}^{\tau,2}=\{\omega_k\omega_l: l\in V^2, k\in V^{1,l}\}.$$ 
We then have 
$$ \begin{aligned}
\|{\mathcal F}^{-1}(\si_{\lambda,\kappa}^{\tau,2} \widehat f \,\,\widehat g\, )\|_{L^1}&\leq \sum_{l\in V^2}   \big\|\sum_{k\in V^{1,l}}b_{kl} \tilde \omega_k\widehat f\, \big\|_{L^2}
 \|\tilde \omega_l\widehat g\, \|_{L^2}
\\& \leq   \bigg( \sum_{l\in V^2_\sigma}  
 \big\|\sum_{k\in V^{1,l}} b_{kl}\tilde \omega_k\widehat f\, \big\|_{L^2}^2 \bigg)^{1/2} \bigg(\sum_{l\in V^2}  \|\tilde \omega_l\widehat g\, \|_{L^2}^2\bigg)^{1/2}.
\end{aligned}
$$
We need to estimate 
$$
\begin{aligned}
 \sum_{l\in V^2}  
 \Big\|\sum_{k\in V^{1,l}} b_{kl}\tilde \omega_k\widehat f\, \Big\|_{L^2}^2
&\leq C \int_{Q} \sum_{l\in V^2}   \sum_{k\in V^{1,l}} B^22^{-2\tau} |\tilde \omega_k|^2 |\widehat f(\xi_1)|^2 d\xi_1\\& \leq C K 2^{\f {2n\lambda}r} B^{2}2^{-2\tau}\|f\|_{L^2}^2,
\end{aligned}
$$
since, by the disjointness of the supports of $\tilde \om_k$,
$\sum_k|\tilde\om_k|^2\le C2^{2n\la/r}$,
and 
the cardinality of $V^2$ is controlled by $K$.

Returning to our estimate, and using orthogonality, we obtain 
\begin{equation}
\label{columns}
 \|{\mathcal F}^{-1}(\si_{\lambda,\kappa}^{\tau,2}\widehat f \,\,\widehat g\,)\|_{L^1}\leq C 
 \|\si\|_{L^r_s}K^{\f 12}  2^{-s\lambda} 2^{-\tau} 2^{ \f{2\lambda n}r} \|f\|_{L^2}\|g\|_{L^2}.
\end{equation}

For any $\tau\le\tau_{\rm max}$ the two inequalities~\eqref{rows} and~\eqref{columns}
are the same due to   $\gamma\leq C2^{\tau r}/K= C2^{\tau r/2}.$ Therefore, we have 
\begin{equation}\label{IMP}
 \|{\mathcal F}^{-1}(\si_{\lambda,\kappa}^{\tau}\widehat f \,\,\widehat g\,)\|_{L^1}\leq C \|\si\|_{L^r_s}2^{(\f r4-1)\tau}2^{\lambda(\f{ 2n} r -s)}   \|f\|_{L^2}\|g\|_{L^2}.
\end{equation}
The right hand side 
has a negative exponent in $\lambda$ since $s>2n/r.$ 

The behavior  in $\tau$ depends on $r$. 
For $1<r<4$ it is a geometric series in $\tau$ and 
hence summing over $0\le\tau\le\tau_{\max}$
and $\la\ge0$  is finite.
However, if $r\geq 4,$ we need to use the following observation:
$$
\sum_{\tau=0}^{\tau_{\rm max} }2^{(\f r4-1)\tau}\leq  
C \tau_{\rm max}2^{(\f r 4 -1)\tau_{\rm max}} 
\le C   \big(\tfrac{ 2n\lambda}r\big)2^{(\f r4-1) \f{2n\lambda}r}.
$$
Therefore, by summing over $\tau$ in \eqref{IMP}  we obtain
$$
\sum_{\tau=0}^{\tau_{\rm max}} \|{\mathcal F}^{-1}(\si_{\lambda,\kappa}^{\tau}\widehat f \,\,\widehat g\,)\|_{L^1}\leq C \|\si\|_{L^r_s} \big(\tfrac{ 2n\lambda}r\big)2^{(\f r4-1) (1+\f{2n\lambda}r)} 2^{\lambda(\f{ 2n} r -s)}   \|f\|_{L^2}\|g\|_{L^2}.
$$
Since $ (2n\lambda/r)2^{(r/4-1)2n\lambda/r} 2^{\lambda(2n/r -s)}= (2n\lambda/r) 2^{\lambda(n/2-s)},$ these estimates form a summable series in $\la$ only if $s>n/2.$

We have $1\leq \kappa\leq C_{n}$ and $\si=\sum_{\lambda=0}^{\infty} \sum_\kappa \si_{\lambda,\kappa}.$ Therefore for $s$ and $r$ related as in $s>\max(2n/r, n/2)$ we have   convergent series, and we obtain the result by summation in $\tau$ first and then in $\la$. 
\end{proof}

\begin{rmk}
We see from the proof (or by an easy dilation argument) that the condition $Q$ is $[-2,2]^n$ 
is not essential and the statement keeps valid when $Q$ is any fixed compact set.
\end{rmk}

\section{The proof of Theorem \ref{MR2}}

\begin{proof}
We use an idea developed in \cite{GHH}, where we consider off-diagonal and  diagonal cases 
separately. For the former we  use the Hardy-Littlewood maximal function and a ``square'' function, and 
for the latter we use use Lemma \ref{MR} in Section~\ref{MLS}.

We introduce   notation needed to study these cases appropriately. We define $\si_j(\xi,\eta)=
\si(\xi,\eta)\wh\psi(2^{-j}(\xi,\eta))$ and write $m_j(\xi,\eta)=\si_j(2^j(\xi,\eta))$. We note
that all $m_j$ are supported in the unit annulus, the dyadic annulus centered at zero with radius 
comparable to $1$, and $\|m_j\|_{L^r_s}\le A$ uniformly in $j$ by assumption \eqref{usb}.

By the discussion in the previous section, for each $m_j$ we have the decomposition
$m_j(\xi,\eta)=\sum_{\kappa}\sum_{\la}\sum_{k,l}b_{k,l}\tilde\om_k(\xi)\tilde \om_l(\eta)=
\sum_{\la}m_{j,\la}$
with $\tilde \om_k\approx 2^{\la n/r}$
and $(\sum_{k,l}|b_{k,l}|^r)^{1/r}\le CA2^{-\la s}$.
Assume that both $\Psi_F$ and $\Psi_M$ are supported in $B(0,N)$ for some large fixed number $N$. We define the off-diagonal parts 
$$
m_{j,\la}^2(\xi,\eta)=\sum_{\kappa}\sum_k\sum_{|l|\le 2\sqrt nN}b_{k,l}\tilde \om_k(\xi)
\tilde \om_l(\eta)
$$
and
$$
m_{j,\la}^3(\xi,\eta)=\sum_{\kappa}\sum_l\sum_{|k|\le 2\sqrt nN}b_{k,l}\tilde \om_k(\xi)
\tilde \om_l(\eta),
$$
then the remainder in the $\la$ level is 
$m_{j,\la}^1(\xi,\eta)=[m_{j,\la}-m_{j,\la}^2-m_{j,\la}^3](\xi,\eta)$ with each wavelet involved  away from the axes.
Moreover for $i=1,2,3$,
we  define $m_j^i=\sum_{\la}m_{j,\la}^i$, $\si_j^i=m_j^i(2^{-j}\cdot)$, $\si^i=\sum_j\si_j^i$.
Notice that $\si$ is equal to the sum $\si^1+\si^2+\si^3$.

\vspace{4pt}
\noindent {\em (i) The Off-diagonal Cases}

We consider the off-diagonal cases $m_{j,\la}^2$ and $m_{j,\la}^3$ first. By symmetry, it suffices
to consider
$$
T_{m_{j,\la}^2}(f,g)(x)=
\int_{\bbr^{2n}}m_{j,\la}^2(\xi,\eta)\wh f(\xi)\wh g(\eta)e^{2\pi i x(\xi+\eta)}d\xi d\eta .
$$
By the definition   $\tilde \om_l=2^{\la n/2}\Psi(2^{\la}x-l)/\|\om_l\|_{L^r}$,
we have
$|(\tilde \om_l\wh g)^{\vee}(x)|\le C2^{\la n/r}M(g)(x)$, where
$M(g)(x)$ is the Hardy-Littlewood maximal function.
Recall the boundedness of $b_{k,l}$ and $\tilde \om_{k}$, we therefore
have
$$
|(\sum_{k}b_{k,l}\tilde \om_k\wh f\, )^{\vee}|
\le 2^{\la(n/r-s)} |(m\wh f\chi_{1/2\le|\xi|\le2})^{\vee}|
$$
with $\|m\|_{L^{\nf}}\le C$. In view of the finiteness of
$N$ and the number of $\kappa$'s, we finally obtain a pointwise control
$$
|T_{m_{j,\la}^2}(f,g)(x)|\le C2^{(2n/r-s)\la} |T_m( f')(x)|M(g)(x),
$$
where $\wh{f'}=\wh f \chi_{1/2\le|\xi|\le2}$.

Observe that
$$
T_{\si_j^2}(f,g)(x)=2^{jn}T_{m_j^2}(f_j,g_j)(2^jx)
$$
with $\wh f_j(\xi)=2^{jn/2}f(2^j\xi)\chi_{1/2\le|\xi|\le2}$
and $\wh g_j(\xi)=2^{jn/2}\wh g(2^j\xi)$. Note that 
we did not define $f_j$ and $g_j$ in  similar ways. By a standard  argument  
using the square function characterization of the Hardy space $H^1$, we control
$\|T_{\si^2}(f,g)\|_{L^1}$
by
\begin{align*}
\Big\|\Big(\sum_j|T_{\si_j^2}(f,g)|^2\Big)^{1/2}\Big\|_{L^1}
= & \Big\|\Big(\sum_j|
2^{jn}T_{m_j^2}(f_j,g_j)(2^j\cdot)|^2\Big)^{1/2}\Big\|_{L^1} \\
\le & \sum_{\la}2^{(2n/r-s)\la}\|g\|_{L^2}\Big(\int \sum_j|\wh f_j(\xi)|^2d\xi\Big)^{1/2}\, .
\end{align*}
Because of the definition of $\wh f_j$, we see that
$$
\int \sum_j|\wh f_j(\xi)|^2d\xi
=\int \sum_j|\wh f(\xi)|^2\chi_{2^{j-1}\le|\xi|\le2^{j+1}}d\xi\le C\|f\|_{L^2}^2.
$$
The exponential decay in $\la$ given by the condition $rs>2n$
then concludes the proof of the off-diagonal cases.

\vspace{4pt}
\noindent {\em (ii) The Diagonal Case}

This case is relatively simple by an argument similar to the 
diagonal part in \cite{GHH},
because we have dealt with the key ingredient in Lemma
\ref{MR}. We give a brief proof here for completeness.
By dilation
we have that
$$
\|T_{\si^1}(f,g)(x)\|_{L^1} \le 
\|\sum_j\sum_{\la}T_{\si_{j,\la}^1}(f,g)\|_{L^1} 
\le  \sum_{\la}\sum_j\|2^{jn}T_{m_{j,\la}^1}(f_j,g_j)(2^j\cdot)\|_{L^1},
$$
where $\wh f_j(\xi)=2^{jn/2}\wh f(2^{jn}\xi)\chi_{C2^{-\la}\le |\xi|\le 2}(\xi)$
because in the support of $m_{j,\la}^1$ we have $C2^{-\la}\le |\xi|\le 2$, 
and $g_j$ is defined similarly. For the last line we apply
Lemma \ref{MR} and obtain,   when $  r\ge4$,
the estimate
$$
\sum_{\la}C \tf{2n\la}r 2^{\la(n/2-s)}\sum_j\|\wh f_j\|_{L^2}\|\wh g_j\|_{L^2}\le
\sum_{\la}C \tf{2n\la}r 2^{\la(n/2-s)}\Big(\sum_j\|\wh f_j\|^2_{L^2}\Big)^{1/2}\Big(\sum_j\|\wh g_j\|^2_{L^2}\Big)^{1/2} .
$$
And when $r< 4$, we have a similar control
$$
\sum_{\la}C 2^{\la(2n/r-s)}\sum_j\|\wh f_j\|_{L^2}\|\wh g_j\|_{L^2}\le
\sum_{\la}C  2^{\la(2n/r-s)}\Big(\sum_j\|\wh f_j\|^2_{L^2}\Big)^{1/2}\Big(\sum_j\|\wh g_j\|^2_{L^2}\Big)^{1/2} .
$$
Observe that
$$
\sum_j\|\wh f_j\|^2_{L^2}=\int|\wh f(\xi)|^2\sum_j\chi_{2^{-\la-j}\le|\xi|\le 2^{1-j}}(\xi)d\xi
\le C\la\|f\|_{L^2}^2,
$$
so  in either case with the restriction $s>\max\{n/2, 2n/r\}$
the  sum over $\la$
is controlled by $\|f\|_{L^2}\|g\|_{L^2}$. Thus we conclude the proof of the diagonal
case and   of Theorem \ref{MR2}.
\end{proof}

\section{Necessary Conditions}
 For a bounded function $\sigma$, let $T_\sigma$ be the $m$-linear multiplier operator with symbol $\si$. 
In this section we obtain examples for $m$-linear multiplier operators that impose restrictions on the indices and the smoothness in order to 
have 
\begin{equation} \label{B33}  
\|T_\si\|_{L^{p_1}(\bbr^n)\times\cdots\times L^{p_m}(\bbr^n)\to L^p(\bbr^n)}\le C\sup_{j\in\mathbb Z}\|\si(2^j\cdot)\wh\Psi\|_{L^r_s(\bbr^{mn})}. 
\end{equation}
These   conditions show in particular that the restriction on $s$ in Theorem \ref{MR2} is necessary.

We first prove Theorem~\ref{MR3} via two counterexamples; these are contained in Proposition \ref{Sq}
and Proposition~\ref{SIS}, respectively.

\begin{prop}\label{Sq}
Under the hypothesis of Theorem~\ref{MR3} 
  we must   have $s\ge (m-1)n/2$.
\end{prop}
\begin{proof}
We use the bilinear case with dimension one to demonstrate the idea first. Then we easily
extend the argument to  higher dimensions.

We fix a Schwartz function $\vp$ with $\hat \vp$ supported in $[-1/100,1/100]$. Let
$\{a_j(t)\}_{j }$ be a sequence of Rademacher functions indexed by positive integers, and
  for $N>1$ define 
$$
\wh{ f_N}(\xi_1)=\sum_{j=1}^N a_j(t_1) \hat \vp (N \xi_1 -j)\, , 
\qq
\wh{ g_N}(\xi_2)=\sum_{k=1}^N a_k(t_2) \hat \vp (N \xi_2 -k).
$$
Let $\phi$ be  a smooth function $\phi$ supported in $[-\tf1{10},\tf1{10}]$ assuming value 
$1$ in $[-\tf1{20},\tf1{20}]$.
We construct the multiplier $\si_N$ of the bilinear operator $T_N$ as follows,
\begin{equation}\label{si_N}
\si_N=\sum_{j=1}^{N}\sum_{k=1}^N a_j(t_1)a_k(t_2) a_{j+k}(t_3)c_{j+k}\phi(N \xi_1-j)\phi(N \xi_2-k),
\end{equation}
where $c_l=1$ when $9N/10\le l\le 11N/10$ and $0$ elsewhere.
Hence 
\begin{align*}
T_N(f_N,g_N)(x)=&\sum_{j=1}^{N}\sum_{k=1}^N a_{j+k}(t_3)c_{j+k}\tf1 {N^2}
\vp(x/N) \vp(x/N)e^{2\pi ix(j+k)/N}\\
=&\sum_{l=2}^{2N}\sum_{k=s_l}^{S_l}a_l(t_3)c_{l}\tf1 {N^2}
\vp(x/N) \vp(x/N)e^{2\pi ixl/N},
\end{align*}
where $s_l=\max(1,l-N)$ and $S_l=\min(N,l-1)$.
We   estimate $\|f_N\|_{L^{p_1}(\bbr)}$, $\|g_N\|_{L^{p_2}(\bbr)}$,
$\|\si_N\|_{L^{r}_s(\bbr^2)}$ and $\|T_N(f_N,g_N)\|_{L^{p}(\bbr)}$.

First we prove that 
$\|f_N\|_{L^{p_1}(\bbr)}\approx N^{1-\tf{p_1}2}$. By Khinchine's inequality we have
\begin{align*}
\int_0^1\|f_N\|_{L^{p_1}}^{p_1}dt_1=&
\int_{\bbr}\int_0^1\big|\sum_{j=1}^Na_j(t_1)\f{\vp(x/N)}Ne^{2\pi ixj/N}\big|^{p_1}dt_1dx\\
\approx&\int_{\bbr}\Big(\sum_{j=1}^N \Big| \f{\vp(x/N)}N  \Big|^2\Big)^{p_1/2}dx\\
\approx&\,\,N^{-p_1/2}\int_{\bbr} \big|\vp(x/N)\big|^{p_1}dx\\
\approx&\,\, N^{1-\tf{p_1}2}.
\end{align*}
Hence
$\|f_N\|_{L^{p_1}(\bbr\times[0,1],\ dxdt)}\approx N^{\tf1{p_1}-\tf12}$.
Similarly 
$\|g_N\|_{L^{p_1}(\bbr\times[0,1],\ dxdt)}\approx N^{\tf1{p_2}-\tf12}$.
The same idea gives that
\begin{align*}
\int_0^1\|T_N(f_N,g_N)\|_{L^{p}}^{p}dt_3\approx&
\int_{\bbr}\Big(\sum_{l=2}^{2N}
\big|c_l(S_l-s_l)\tf1 {N^2}
\vp^2(x/N)e^{2\pi ixl/N}\big|^2\Big)^{p/2}dx\\
\approx&
\int_{\bbr} \Big(\sum_{l=9N/10}^{11N/10}(S_l-s_l)^2\Big)^{p/2}
\tf1 {N^{2p}}
|\vp(x/N)|^{2p}dx\\
\approx&\,\,N^{\tf{3p}2-2p}\int_{\bbr}|\vp(x/N)|^{2p}dx\\
\approx&\,\, N^{1-\tf{p}2}.
\end{align*}
In other words we showed that 
$\|T_N(f_N,g_N)\|_{L^{p}(\bbr\times[0,1],\ dxdt)}\approx N^{\tf1{p}-\tf12}$.

As for   $\si_N$, we have the following result whose proof can be found in \cite[Lemma 4.2]{P1}. 

\begin{lm}\label{SN} 
For the multiplier $\si_N$ defined in \eqref{si_N} and any $s\in(0,1)$, there exists a constant $C_s$ such that 
\begin{equation}
\|\si_N\|_{L^r_s(\bbr^2)}\le C_sN^s.
\end{equation}
\end{lm}

Apply \eqref{B1} to $f_N,\ g_N$ and $T_N$ defined above and integrate with respect to 
$t_1$, $t_2$ and $t_3$ on both sides, we have
$$
 \bigg(\int_0^1\int_0^1\int_0^1\|T_N(f_N,g_N)\|_{L^p}^pdt_3dt_1dt_2\bigg)^{1/p}\le   
  C_sN^s 
\bigg(\int_0^1\|f_N\|^p_{L^{p_1}}dt_1\int_0^1\|g_N\|_{L^{p_2}}^pdt_2\bigg)^{1/p},
$$
which combining the estimates obtained on $f_N$, $g_N$ and $T_N(f_N,g_N)$ above
implies
$$
N^{\tf1{p}-\tf12}\le C_s N^s
N^{\tf1{p_1}-\tf12}N^{\tf1{p_2}-\tf12},
$$
so we automatically have $N^{1/2}\le C_s N^s$, which
is true when $N$ goes to $\nf$ only if $s\ge 1/2$.

We now discuss the case $m\ge2$ and $n=1$. We use
for $1\le k\le m$ 
$$
\widehat {f_k}(\xi_k)=\sum_{j=1}^N a_j(t_k) \widehat \vp (N \xi_k -j),
$$
and
$$
\si_N=\sum_{j_1=1}^{N}\cdots\sum_{j_m=1}^N a_{j_1}(t_1)\cdots a_{j_m}(t_m) a_{j_1+\cdots+j_m}(t_{m+1})c_{j_1+\cdots+j_m}\prod_{k=1}^m\phi(N \xi_k-j_k).
$$
By an argument similar to the case $m=2$ and $n=1$, we have
$$
\|f_k\|_{L^{p_k}(\bbr\times[0,1],\ dxdt)}\approx N^{\tf1{p_k}-\tf12},
$$
$\|\si_N\|_{L^r_s}\le C N^s$
and
\begin{equation}\label{mo}
\|T(f_1,\dots,f_m)\|_{L^p(\bbr)}\approx N^{\f 1 p-\f 1 2},
\end{equation}
hence we obtain that $s\ge (m-1)/2$.

For the higher dimensional cases, we define
$$
F_k(x_1,\dots, x_n)=\prod_{\tau=1}^nf_k(x_\tau) ,
$$
and 
$\si(\xi_1,\dots, \xi_n)=\prod_{\tau=1}^n\si_N(\xi_\tau)$,
then $\|F_k\|_{L^{p_k}}\approx N^{n(\tf1{p_k}-\tf12)}$,
$\|\si\|_{L^r_s}\le CN^s$,
and 
$$
\|T(F_1,\dots, F_m)\|\approx N^{n(\tf1{p}-\tf12)} . 
$$
We therefore obtain the 
restriction $s\ge (m-1)n/2$.
\end{proof}

\begin{prop}\label{SIS}  
Under the hypothesis of Theorem~\ref{MR3} 
  we must   have   $ s\ge mn/r$.
\end{prop}

\begin{proof}
Let $\vp$ and $\phi$ be as in Proposition \ref{Sq}.
Define $\wh f_j(\xi_j)=\wh\vp(N(\xi_j-a))$ with $|a|=1$, and $\si(\xi,\dots,\xi_m)=\prod_{j=1}^m
\phi(N(\xi_j-a))$, then a direct calculation gives 
$\|f_j\|_{L^{p_j}(\rn)}\approx N^{-n+n/p_j}$ and $\|\si\|_{L^r_s(\bbr^{mn})}
\le CN^sN^{-mn/r}$. 
Moreover, 
$$
T_{\si}(f_1,\dots, f_m)(x)=N^{-mn}(\vp(x/N)e^{2\pi ix\cdot a})^m .
$$
We can therefore obtain that
$\|T_\si(f_1,\dots ,f_m)\|_{L^p(\rrn)}\approx N^{-mn+n/p} CN^sN^{-mn/r}$.
Then we come to the inequality $N^{-mn+n/p}\le CN^sN^{-mn/r}\prod_jN^{-n+n/p_j}$,
which forces $s-mn/r\ge0$ by letting $N$ go to infinity. 
\end{proof}

Next, we obtain from \eqref{B33} the  restrictions for the indices $p_j$ claimed in Theorem~\ref{PL}.

\begin{proof} (Theorem~\ref{PL})
By symmetry it suffices to consider the case $I=\{1,2,\ \dots,\ k\}$
with $k\in\{0,1,\dots ,m\}$ and the explanation $I=\emptyset$ when $k=0$.
Define  for $\xi\in\mathbb R$
$$
\wh {f_N}(\xi)=\sum_{j=-N}^N\wh\vp(N\xi-j)a_j(t), \quad\quad
\wh {g_N}(\xi)=\sum_{j=-N}^N\wh\vp(N\xi-j),
$$
and
\begin{align*}
&
\si_N(\xi_1,\dots,\xi_m) \\
&=\sum_{j_1=-N}^N\cdots\sum_{j_m=-N}^N
a_{j_1+\cdots+j_m}(t)c_{j_1+\cdots+j_m}a_{j_1}(t_1)\cdots a_{j_k}(t_k)
\phi(N\xi_1-j_1)\cdots \phi(N\xi_m-j_m).
\end{align*}
The idea is that in this setting 
if we take the first $k$ functions as $f_N$ and the remaining as $g_N$, we have
\begin{align*}
&T_{\si_N}(\overbrace {f_N,\dots,f_N}^{\textup{$k$ terms}},\overbrace {g_N,\dots,g_N}^{\textup{$m-k$ terms}})(x) \\ =
&\sum_{j_1=-N}^N\cdots\sum_{j_m=-N}^N
a_{j_1+\cdots+j_m}(t)c_{j_1+\cdots+j_m}
N^{-m}[\vp(  x/N)]^me^{2\pi ix(j_1+\cdots+j_m)/N}.
\end{align*}
This expression is independent of $k$ and by \eqref{mo}
we know 
$$
\|T_{\si_N}(f_N,\dots,f_N,g_N,\dots,g_N)\|_{L^p} \approx N^{1/p-1/2} .
$$
Previous calculations show also $\|f_N\|_{L^{p_i}}\approx C_{p_i}N^{1/p_i-1/2}$
and $\|\si_N\|_{L^r_s}\le CN^s$. Lemma 4.3 in \cite{P1} gives that
$\|g_N\|_{L^{p_i}}\le C_{p_i}$ for $p_i\in(1,\nf]$.
Consequently, we have 
$$
N^{\f 1 p-\f 1 2}\le CN^{\sum_{i=1}^k(\f1{p_i}-\f12)}N^s
$$
and this verifies our conclusion when $n=1$.

For the higher dimensional case, we just use the tensor products and $\si$
similar to what we have in Proposition \ref{Sq}, and thus conclude the proof.
\end{proof}

Notice that when $k=m$, Theorem~\ref{PL} coincides with Proposition \ref{Sq}.

\end{document}